\documentclass[11pt,a4paper,reqno]{amsart}
\usepackage{titlesec}
\usepackage{mathrsfs}
\usepackage{amsmath,amssymb}
\usepackage{amsfonts}
\usepackage{latexsym,bm}
\usepackage{amsthm}
\usepackage[normalem]{ulem}
\usepackage{cite}
\usepackage{amssymb}
\usepackage{amssymb,amscd}
\usepackage{amsbsy}
\usepackage{fancyhdr,graphicx}
\usepackage{indentfirst}
\usepackage{graphics,color}
\usepackage{epsfig}
\usepackage{xcolor}
\usepackage{overpic}
\usepackage{caption}
\usepackage{subcaption}
\usepackage{multirow}
\usepackage{diagbox}
\usepackage{array}
\usepackage{makecell}
\usepackage{arydshln}
\usepackage{booktabs}
\usepackage{threeparttable}
\usepackage{tikz} 
\usetikzlibrary{patterns}
\usepackage{hyperref}

\usepackage{enumerate}
\titleformat{\section}{\normalfont\Large\bfseries}{\thesection}{1em}{}
\titleformat{\subsection}{\normalfont\large\bfseries}{\thesubsection}{0.5em}{}
\titleformat{\subsubsection}{\normalfont\normalsize\bfseries}{\thesubsubsection}{0.5em}{}

\newtheorem {theorem} {Theorem}
\newtheorem {proposition} {Proposition}
\newtheorem {corollary} {Corollary}
\newtheorem {lemma}  {Lemma}

\newtheorem {remark} {Remark}
\newtheorem {problem} {Problem}
\theoremstyle{remark}
\theoremstyle{definition}

\DeclareMathOperator{\sgn}{sgn}

\newpage
\makeatletter
\def\@settitle{\begin{center}\normalfont\Large\bfseries\@title\end{center}\par\vskip 0.5em}
\makeatother
\begin{document}

\title[Limit cycles of Abel equations]
{
Uniqueness of Limit Cycles for Abel Equations and Planar Systems}

\author[J. L. Bravo]{José Luis Bravo Trinidad}
\address{Departamento de Matem\'{a}ticas, Universidad de Extremadura, 06006 Badajoz, Spain}
\email{trinidad@unex.es}

\author[R. Tian]{Renhao Tian}
\address{School of Mathematics (Zhuhai), Sun Yat-sen University, Zhuhai 519082, P. R. China}
\email{tianrh5@mail.sysu.edu.cn}

\author[Y. Zhao]{Yulin Zhao}
\address{School of Mathematics (Zhuhai), Sun Yat-sen University, Zhuhai 519082, P. R. China}
\email{mcszyl@mail.sysu.edu.cn}

\begin{abstract}
We study positive limit cycles of the Abel equation
$$\dot{x}=A(t)x^{3}+B(t)x^{2},$$
where $A$ and $B$ are $2\pi$-periodic functions. 
By developing a half-period convexity method, we establish a new criterion ensuring that the equation has at most one positive limit cycle.
Moreover, if such a limit cycle exists, it is hyperbolic.
As applications, we consider two classes of planar polynomial systems that can be reduced to Abel equations.
For a subfamily of cubic rigid systems, we prove that there is at most one limit cycle, thereby giving a partial answer to the first problem in Gasull's list of 33 open problems.
For quadratic systems, we derive a new uniqueness criterion for limit cycles surrounding the origin and provide explicit parameter regions under which the criterion is satisfied.
\end{abstract}

\date{}
\maketitle
\section{Introduction and Main Results}
The study of periodic solutions of Abel equations occupies a special place in the qualitative theory of differential equations, particularly in connection with the second part of Hilbert's 16th problem.
Indeed, by means of suitable changes of variables, the study of limit cycles of several classes of planar polynomial systems can be reduced to the study of periodic solutions of Abel equations.
Classical examples include quadratic systems \cite{Cherkas1977,Lins}, cubic systems \cite{DevlinLloydPearson1998}, rigid systems \cite{Gasull2021,GasullProhensTorregrosa2005}, and systems with homogeneous
or quasi-homogeneous nonlinearities \cite{CarbonellLlibre1988,CollGasullProhens1997,HuangLiang2018,HuangLiangLlibre2018}.
A systematic study of polynomial planar systems that can be reduced to Abel equations was recently developed in \cite{AlvarezBravoSanchez2022}.

Motivated by these connections, in this paper we consider the Abel equation
\begin{equation}\label{two-term-Abel}
    \dot{x}=A(t)x^3+B(t)x^2,
\end{equation}
where $t,x\in\mathbb{R}$ and $A$, $B$ are $2\pi$-periodic functions.
For $t_0,x_0\in\mathbb{R}$, let $x(t;x_0)$ denote the solution of equation \eqref{two-term-Abel} satisfying $x(t_0;x_0)=x_0$.
The solution $x(t;x_0)$ is called \textit{a $2\pi$-periodic solution} (or simply \textit{a periodic solution}) if it is defined for every $t\in\mathbb{R}$ and satisfies $x(t_0+2\pi;x_0)=x_0$ (therefore, it satisfies $x(t+2\pi;x_0)=x(t;x_0)$ for every $t\in\mathbb{R}$).
Clearly, $x=0$ is always a periodic solution of \eqref{two-term-Abel}.
As usual, an isolated periodic solution will be called a \textit{limit cycle}.
We are interested in the number of positive limit cycles, namely, isolated periodic solutions satisfying $x(t;x_{0})>0$.
 
The problem of determining the number of limit cycles of Abel equations has been extensively studied. 
Nevertheless, even for relatively simple families of periodic coefficients, 
obtaining effective upper bounds remains a difficult problem. 
In fact, Lins Neto \cite{Lins} proved that,
for every $k\in\mathbb N$, 
there exist periodic functions $A$ and $B$ such that equation \eqref{two-term-Abel} has at least $k$ limit cycles.
This shows that no uniform finite bound exists over the full class of periodic coefficients. 
Therefore, meaningful upper bound problems require either restrictions on the complexity of the coefficients or additional structural assumptions.

A number of classical upper bound criteria are based on sign conditions on the coefficients. 
In particular, Pliss~\cite{Pliss} and Gasull and Llibre \cite{GL} obtained classical bounds for Abel equations when one of the coefficients $A$ or $B$ has a fixed sign.
A significant extension in this direction was given by \'{A}lvarez, Gasull and Giacomini \cite{AGG}, 
who proved that if there exists a nonzero vector $(\lambda,\mu)\in\mathbb R^2$ such that $\lambda A(t)+\mu B(t)$ does not change sign and is not identically zero, 
then \eqref{two-term-Abel} has at most one nonzero limit cycle. 
More recently, Huang, Tian and Zhao \cite{Huang2026} obtained a different criterion based on Chebyshev theory. 
More precisely, if $A$ and $B$ belong to the linear span of an ET-system of dimension three containing a positive function, 
then  equation \eqref{two-term-Abel} has at most two nonzero limit cycles, counted with multiplicity.
Other approaches based on symmetries, geometric criteria, invariant curves, and special configurations of sign-changing coefficients have also been developed \cite{AlvarezBravoFernandez2013,AlvarezBravoFernandez2015,HuangLiang2017,HuangLiang2020,HuangZhao2012,BravoCalderonFernandez2021,YuChenLiu2024,Yu2024}.

Despite these developments, the direct application of the existing criteria to planar polynomial systems remains rather limited.
Indeed, for many Abel equations arising from planar systems, the coefficient functions themselves change sign over one period, and there need not exist a nontrivial constant linear combination of them having a fixed sign.
Moreover, these coefficients may have several zeros, which also makes criteria based on low-dimensional Chebyshev theory difficult to apply.
This phenomenon already appears in classical reductions of rigid and quadratic systems to Abel equations.

These observations motivate us to develop a different approach, which we call \textit{the half-period convexity method}. 
The main idea is to rewrite the stability condition for a positive periodic solution as a comparison problem on a single half-period and to exploit the convexity or concavity of the quotient $A/B$ after a suitable reparametrization of the independent variable.
This leads to a new criterion for the uniqueness of positive limit cycles of \eqref{two-term-Abel}.

Our main result is the following.

\begin{theorem}\label{criterion}
Consider the Abel equation \eqref{two-term-Abel}, where $A, B\in C(\mathbb{R})$  are $2\pi$-periodic. 
Assume that
\begin{enumerate}
\item[(i)] $A(t+\pi)=A(t)$, $B(t+\pi)=-B(t)$, and
      $B(t)>0$ for $t\in(0,\pi)$;
\item[(ii)] $A(t)$ has exactly two zeros $t_1<t_2$ in $(0,\pi)$;
\item[(iii)] Set
      $$
      z(t)=\int_0^t B(s)\,ds,\qquad K(t)=\frac{A(t)}{B(t)},\quad t\in(0,\pi).
      $$
      Since $B(t)>0$ on $(0,\pi)$, the function $z$ is strictly increasing.
    Let $t=t(z)$ denote its inverse and define
      $$
      \widehat K(z):=K(t(z))\qquad z\in\bigl(z(0),z(\pi)\bigr).
      $$
    One of the following alternatives holds:
      \begin{enumerate}
      \item[(iii.1)] $A<0$ on $(t_1,t_2)$ and $A>0$ on $(0,t_1)\cup(t_2,\pi)$, and $\widehat K$ is strictly convex on $\left(z(0),z(\pi)\right)$;
      \item[(iii.2)] $A>0$ on $(t_1,t_2)$, $A<0$ on $(0,t_1)\cup(t_2,\pi)$, and $\widehat K$ is strictly concave on $\left(z(0),z(\pi)\right)$.
      \end{enumerate}
\end{enumerate}
Then equation \eqref{two-term-Abel} has at most one positive limit cycle.
If such a limit cycle exists, it is hyperbolic;
moreover, it is unstable under \textup{(iii.1)} and stable under \textup{(iii.2)}. 
\end{theorem}

We next illustrate the applicability of Theorem~\ref{criterion} through two different classes of planar polynomial systems whose limit cycle problems can be reduced to Abel equations.

As a first application, we consider planar polynomial rigid systems, introduced by Conti \cite{Conti1994}.
After a linear change of variables and a rescaling of time, such systems can be written as
\begin{equation}\label{rigid equation0}
\begin{cases}
\dot{u}=-v+u\left(\sum_{j=0}^{n}F_j(u,v)\right),\\[10pt]
\dot{v}=\phantom{-}u+v\left(\sum_{j=0}^{n}F_j(u,v)\right),
\end{cases}
\end{equation}
where $(u,v)\in\mathbb{R}^2$, and each $F_j$ is a homogeneous polynomial of degree $j$.
Introducing polar coordinates, we transform the above system into the following Abel-type equation
\begin{equation*}
\frac{dr}{d\theta}
=
\sum_{j=0}^{n}F_j(\cos\theta,\sin\theta)r^{j+1}.    
\end{equation*}
Thus, the limit cycle problem for system \eqref{rigid equation0} is naturally reduced to the study of positive limit cycles of the above equation.

It is already known that when $n=0,1$, system \eqref{rigid equation0} has no limit cycles; see \cite{Gasull2021}.
Hence, the first nontrivial case for the limit cycle problem is $n=2$.
The corresponding cubic rigid system can be written as
\begin{equation}\label{rigid equation1}
\begin{cases}
\dot u=-v+u(a+bu+cv+du^2+euv+fv^2),\\
\dot v=\phantom{-}u+v(a+bu+cv+du^2+euv+fv^2),
\end{cases}
\end{equation}
where $d^2+e^2+f^2\neq0$.
Its radial equation is
\begin{equation}\label{Abel0}
\frac{dr}{d\theta}=ar+(b\cos\theta+c\sin\theta)r^2+(d\cos^2\theta+e\cos\theta\sin\theta+f\sin^2\theta)r^3.
\end{equation}
The number of limit cycles of system \eqref{rigid equation1} is not known in general. 
Gasull, Prohens and Torregrosa \cite{GasullProhensTorregrosa2005} proved, by means of a degenerate Andronov-Hopf bifurcation, that two limit cycles can occur.
They also established uniqueness when $\Delta:=e^2-4df\leq0$.
Further uniqueness results were obtained by Gasull and Torregrosa \cite{GasullTorregrosa2005} for the subfamily $b=c=0$
and more recently by Gon\c{c}alves et al. \cite{GoncalvesEtAl2026} in a small parameter regime, and a non-existence result was obtained by  Bravo and Torregrosa~\cite{BT} for the subfamily $a=d=f=0$.

In 2021, Gasull \cite{Gasull2021} proposed a list of 33 open problems in low-dimensional dynamical systems, 
the first of which concerns the maximum number of limit cycles of cubic rigid systems.
In the present notation, the problem can be stated as follows.

\begin{problem}
Is 2 the maximum number of limit cycles of system \eqref{rigid equation1}?
\end{problem}
We provide a partial answer to this problem by considering the natural subfamily $a=0$. 
In this case, equation \eqref{Abel0} reduces to
\begin{equation}\label{Abel1}
\frac{dr}{d\theta}=(b\cos\theta+c\sin\theta)r^2+(d\cos^2\theta+e\cos\theta\sin\theta+f\sin^2\theta)r^3,
\end{equation}
which is precisely of the form \eqref{two-term-Abel}.
Moreover, its coefficients possess the half-period structure required by Theorem~\ref{criterion}. 
By applying our criterion, together with the necessary parameter reductions, we obtain the following result.

\begin{theorem}\label{mainresult-rigid}
For $a=0$, the cubic rigid system \eqref{rigid equation1} has at most one limit cycle. 
Moreover, this upper bound is sharp.
\end{theorem}
Therefore, Theorem~\ref{mainresult-rigid} gives a partial answer to the first problem in \cite{Gasull2021}.

\medskip 

As a second application, we consider quadratic systems with an equilibrium point whose linearization is a center, which can be written in the form
\begin{equation}\label{quadratic-system}
\begin{cases}
\dot x=-y+P_2(x,y),\\
\dot y=\phantom{-}x+Q_2(x,y),
\end{cases}
\end{equation}
where $P_2$ and $Q_2$ are homogeneous polynomials of degree two. As we will show later, limit cycles of \eqref{quadratic-system} surrounding the origin are in one-to-one correspondence with positive limit cycles of the Abel equation
\begin{equation}\label{quadratic-Abel}
\rho'=A(\theta)\rho^3+B(\theta)\rho^2,
\end{equation}
where $A,B$ are homogenous trigonometric polynomials of degrees $6$ and $3$, respectively, $B(0)=0$, and $B(\theta)>0$ for sufficiently small $\theta>0$.

This correspondence allows us to translate Theorem~\ref{criterion} into a criterion stated directly in terms of the quadratic system.
Although this criterion can be expressed entirely in terms of the coefficients of $P_2$ and $Q_2$, the resulting algebraic conditions are rather cumbersome. 
Instead, its formulation in terms of the functions $A$ and $B$ above provides a convenient way to verify the conditions of Theorem~\ref{criterion}. 
For later reference, define
\begin{equation*}
\mathcal C(\theta)
=
B(\theta)^2A''(\theta)
-A(\theta)B(\theta)B''(\theta)
-3B(\theta)A'(\theta)B'(\theta)
+3A(\theta)(B'(\theta))^2.    
\end{equation*}
We then obtain the following result.
\begin{theorem}\label{quadratic-criterion}
Assume that 
\begin{enumerate}
\item[(i)]
$B(\theta)>0$ for
$\theta\in(0,\pi)$;
\item[(ii)]
$A$ has exactly two simple zeros in $[0,\pi]$;
\item[(iii)]$A(0)\mathcal C(\theta)>0$ for $\theta\in(0,\pi)$.
\end{enumerate}
Then system \eqref{quadratic-system} has at most one limit cycle surrounding the origin. 
If it exists, it is hyperbolic. 
Moreover, it is unstable when $A(0)>0$, and stable when $A(0)<0$.
\end{theorem}

In particular, we obtain the following explicit two-parameter family.

\begin{corollary}\label{mainresult-quadratic}
Consider the quadratic system
\begin{equation}\label{explicit-quadratic}
\begin{cases}
\dot x=-y+ax^2+2bxy+ay^2,\\
\dot y=\phantom{-}x+x^2+2axy+3by^2.
\end{cases}
\end{equation}
If $|a|\leq \frac13$ and $1\leq b\leq \frac32$,
then system \eqref{explicit-quadratic} has at most one limit cycle surrounding the origin, and this upper bound is sharp.
Moreover, if such a limit cycle exists, its stability is determined by the sign of $a$.
It is stable if $a>0$; it is unstable if $a<0$.
For $a=0$, the origin is a center.
\end{corollary}

The rest of the paper is organized as follows.
In Section~2, we prove Theorem~\ref{criterion}.
Section~3 is devoted to the application to cubic rigid systems and the
proof of Theorem~\ref{mainresult-rigid}.
In Section~4, we derive the quadratic-system criterion and prove
Theorem~\ref{quadratic-criterion}, together with its explicit
applications.

\section{Uniqueness of Limit Cycles for Abel Equations}
The aim of this section is to prove Theorem~\ref{criterion}. 
The proof is based on the stability of positive periodic solutions. 
We first restrict the stability analysis to a single half-period. 
We then use the strict convexity/concavity of $\widehat{K}$ to show that every positive periodic solution is hyperbolic and has the same stability.
Therefore, there is at most one positive limit cycle.

\begin{remark}
It suffices to consider case (iii.1). 
Indeed, if (iii.2) holds, then under the time-reversal transformation $t\mapsto -t$, equation~\eqref{two-term-Abel} becomes
\begin{equation*}
\frac{dx}{dt}=-A(-t)x^3-B(-t)x^2:=A^*(t)x^3+B^*(t)x^2,
\end{equation*}
For $0<t<\pi$, we have 
\[{A^*(t)=-A(-t)=-A(\pi-t),\quad B^*(t)=-B(-t)=B(\pi-t)>0}.\] Thus, the two zeros of $A^*$ are $t_1^*=\pi-t_2$ and $t_2^*=\pi-t_1$, and $A^*$ has the sign configuration in (iii.1).
Under the time reversal $t\mapsto -t$, we have \[z^*(t)=\int_0^t B^*(s)\,ds=
\int_0^t B(\pi-s)\,ds=z(\pi)-z(\pi-t),\] 
\[K^*(t)=\frac{A^*(t)}{B^*(t)}
=-\frac{A(-t)}{B(-t)}=-K(-t),\] so that
\[
\widehat{K}^*(z^*(t)) = K^*(t) = -K(-t) = -\widehat{K}(z(-t)).
\]
Therefore
\[
\widehat{K}^*(z) = -\widehat{K}(z(\pi)-z).
\]
Hence the strict concavity in (iii.2) is transformed into the strict convexity in (iii.1).   
\end{remark} 

Therefore, in what follows, we only assume that (iii.1) holds. In particular, $A(t)>0$ for $t\in(0,t_1)\cup(t_2,\pi)$ and $A(t)<0$ for $t\in(t_1,t_2)$.

Let $\mathcal D$ denote the set of positive initial values $x_0$ for which the corresponding solution of~\eqref{two-term-Abel} is defined on $[0,2\pi]$. 
For $x_{0}\in\mathcal D$, denote by $x(t;x_0)$ the solution satisfying $x(0;x_0)=x_0$, and define the displacement map by
\begin{equation*}
D(x_0)=x(2\pi;x_0)-x_0.
\end{equation*}
Then $x(t;x_0)$ is a $2\pi$-periodic solution if and only if $D(x_0)=0$. 
Moreover, its stability is determined by the sign of $D'$, being stable if $D'(x_0)<0$ and unstable if $D'(x_0)>0$. 
Differentiating the flow with respect to the initial value and integrating the resulting variational equation, we obtain
\begin{equation}\label{eq:D'}
D'(x_0)=\exp\left(\int_{0}^{2\pi}\left(3A(t)x^2(t;x_0)+2B(t)x(t;x_0)\right)\,dt\right)-1.
\end{equation}

We next obtain a simpler expression to determine the stability of periodic solutions. 
To that end, let us fix a positive periodic solution $x(t):=x(t;x_0)$ and set $y(t):=\frac{1}{x(t)}$.

\begin{proposition}\label{prop:stability1}
Assume that condition (i) of Theorem~\ref{criterion} holds, and that $x(t)$ is a positive periodic solution of~\eqref{two-term-Abel}. Then, for any $t_0\in\mathbb{R}$,
\begin{equation*}
\sgn\left(D'(x_0)\right)
=
\sgn\left(
\int_{t_0}^{t_0+\pi}A(t)\Big(x(t)+x(t+\pi)\Big)^2\,dt
\right).
\end{equation*}
\end{proposition}

\begin{proof}
Using that $x(t)$ is $2\pi$-periodic and~\eqref{eq:D'}, for any $t_0\in\mathbb{R}$ the stability of $x(t)$ is determined by
\begin{equation*}
\begin{split}
\sgn\left(D'(x_0)\right)
&=\sgn\left(\int_{0}^{2\pi}\left(3A(t)x^2(t)+2B(t)x(t)\right)\,dt\right)\\
&=\sgn\left(\int_{t_0}^{t_0+2\pi}\left(3A(t)x^2(t)+2B(t)x(t)\right)\,dt\right).
\end{split}
\end{equation*}
Denote
\begin{equation*}
I:=\int_{t_0}^{t_0+2\pi}\left(3A(t)x^2(t)+2B(t)x(t)\right)\,dt.
\end{equation*}
Since $x(t)$ is $2\pi$-periodic, 
\begin{equation*}
\int_{t_0}^{t_0+2\pi} \left(A(t)x^2(t)+B(t)x(t)\right) \, dt=\int_{t_0}^{t_0+2\pi} \frac{x'(t)}{x(t)} \, dt = \ln(x(t_0+2\pi))-\ln(x(t_0))=0.
\end{equation*}
Therefore, 
\begin{equation*}
I
=\int_{t_0}^{t_0+2\pi} 3A(t)x^2(t)+2B(t)x(t)-2\frac{x'(t)}{x(t)}\,dt
= \int_{t_0}^{t_0+2\pi}A(t)x^2(t)\,dt.
\end{equation*}
Since $A(t+\pi)=A(t)$,  we have
\begin{equation*}
    \begin{split}
        I=&\int_{t_0}^{t_0+\pi}A(t)x^2(t)\,dt+\int_{t_0+\pi}^{t_0+2\pi}A(t)x^2(t)\,dt\\
         =&\int_{t_0}^{t_0+\pi}A(t)x^2(t)\,dt+\int_{t_{0}}^{t_0+\pi}A(t+\pi)x^2(t+\pi)\,dt\\
         =&\int_{t_0}^{t_0+\pi}A(t)\left(x^2(t)+x^2(t+\pi)\right)\,dt\\
         =&\int_{t_0}^{t_0+\pi}A(t)\left(\frac{1}{y^2(t)}+\frac{1}{y^2(t+\pi)}\right)\,dt\\
         =&\int_{t_0}^{t_0+\pi}A(t)\left(\frac{1}{y(t)}+\frac{1}{y(t+\pi)}\right)^{2}\,dt-\int_{t_0}^{t_0+\pi}\frac{2A(t)}{y(t)y(t+\pi)}\,dt\\
    \end{split}
\end{equation*}
Define $r(t)=y(t)+y(t+\pi)>0$. Since $y$ is $2\pi$-periodic, $r$ is $\pi$-periodic. 
Moreover, using $B(t+\pi)=-B(t)$ and $A(t+\pi)=A(t)$, we obtain
\begin{equation}\label{eq:difr}
\begin{split}
r'(t)
&=y'(t)+y'(t+\pi)\\
&=\left(-B(t)-\frac{A(t)}{y(t)}\right)
+\left(-B(t+\pi)-\frac{A(t+\pi)}{y(t+\pi)}\right)\\
&=-\frac{A(t)r(t)}{y(t)y(t+\pi)}.
\end{split}
\end{equation}
Hence
\begin{equation*}
\int_{t_0}^{t_0+\pi}\frac{A(t)}{y(t)y(t+\pi)}\,dt
=-\int_{t_0}^{t_0+\pi}\frac{r'(t)}{r(t)}\,dt
=-\ln\frac{r(t_0+\pi)}{r(t_0)}=0.
\end{equation*}
Thus, we conclude that for every $t_0$,
\begin{equation*}
\begin{split}
\sgn\left(D'(x_0)\right) =&\sgn\left(\int_{t_0}^{t_0+\pi}A(t)\left(\frac{1}{y(t)}+\frac{1}{y(t+\pi)}\right)^{2}dt\right) \\
=&\sgn\left(\int_{t_0}^{t_0+\pi}A(t)\Big(x(t)+x(t+\pi)\Big)^2 dt\right).
\end{split}
\end{equation*}
\end{proof}

In what follows, we assume that conditions (i), (ii) and (iii.1) of Theorem~\ref{criterion} hold. By~\eqref{eq:difr}, $r'(t)<0$ on $(0,t_1)$, $r'(t)>0$ on $(t_1,t_2)$ and $r'(t)<0$ on $(t_2,t_1+\pi)$. Define
\begin{equation*}
r_1=r(t_1),\qquad r_2=r(t_2),\qquad r_*=r(\pi)=r(0).
\end{equation*}
Then $r_1<r_*<r_2$. 
See Figure \ref{fig:r}.
Moreover, let $\tau_0$, $\tau_1$ and $\tau_2$ denote the inverse functions of $r(t)$ on $[0,t_1]$, $[t_1,t_2]$ and $[t_2,t_1+\pi]$, respectively. 
Thus,
\begin{equation*}
\tau_0:[r_1,r_*]\to[0,t_1],\qquad
\tau_1:[r_1,r_2]\to[t_1,t_2],\qquad
\tau_2:[r_1,r_2]\to[t_2,t_1+\pi].
\end{equation*}
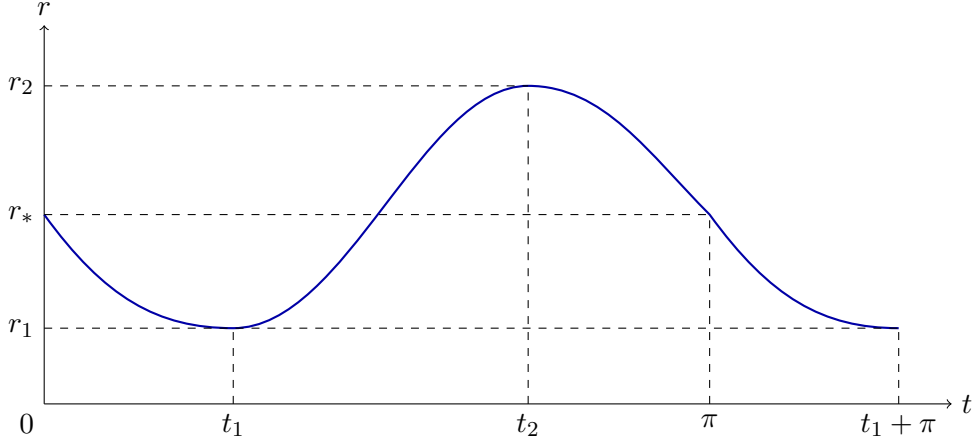
\begin{figure}[t]
\centering
\begin{tikzpicture}[x=1cm,y=1cm]
 \draw[->] (0,0)--(12,0) node[right] {$t$};
 \draw[->] (0,0)--(0,5) node[above] {$r$};
 \draw[blue!65!black,thick]
 (0,2.5).. controls (0.8,1.4) and (1.5,1)..(2.5,1)
 .. controls (4,1) and (5,4.2)..(6.4,4.2)
 .. controls (7.4,4.2) and (8,3.3)..(8.8,2.5)
 .. controls (9.6,1.4) and (10.3,1)..(11.3,1);
 \draw[dashed] (0,1)--(11.3,1);
 \draw[dashed] (0,2.5)--(8.8,2.5);
 \draw[dashed] (0,4.2)--(6.4,4.2);
 \foreach \x/\h/\lab in {2.5/1/t_1,6.4/4.2/t_2,8.8/2.5/\pi,11.3/1/t_1+\pi}
 {\draw[dashed] (\x,0)--(\x,\h);\node[below] at (\x,0) {$\lab$};}
 \node[left] at (0,1) {$r_1$};
 \node[left] at (0,2.5) {$r_*$};
 \node[left] at (0,4.2) {$r_2$};
 \node[below left] at (0,0) {$0$};
\end{tikzpicture}
\caption{Sketch of the function $r(t)$.}
\label{fig:r}
\end{figure}
Define $R(t)=y(t)-y(t+\pi)$. 
Then
\begin{equation*}
r(t+\pi)=r(t),\qquad R(t+\pi)=-R(t),\qquad y(t)=\frac{r(t)+R(t)}{2}.
\end{equation*}
A direct computation gives
\begin{equation}\label{eq:difR}
    \begin{split}
        \frac{dR(t)}{dt}=&\frac{dy(t)}{dt}-\frac{dy(t+\pi)}{dt}\\
        =&-\left(B(t)+\frac{A(t)}{y(t)}\right)+\left(B(t+\pi)+\frac{A(t+\pi)}{y(t+\pi)}\right)\\
        =&-2 B(t) + A(t)\left( \frac{1}{y(t+\pi)}-\frac{1}{y(t)} \right).
    \end{split}    
\end{equation}
Combining~\eqref{eq:difr} and~\eqref{eq:difR}, we obtain

\begin{equation}\label{eq:rRmonotone}
    \begin{split}
\frac{d}{dt}\big(r(t)R(t)\big)
&=r'(t)R(t)+r(t)R'(t)\\
&=-\frac{A(t)r(t)R(t)}
        {y(t)y(t+\pi)}
+r(t)\left(
-2B(t)
+\frac{A(t)R(t)}
       {y(t)y(t+\pi)}
\right)\\
&=-2B(t)r(t).
\end{split}
\end{equation}
This shows that $\frac{d}{dt}\big(r(t)R(t)\big)<0$ for $t\in(0,\pi)$, and hence $r(t)R(t)$ is strictly decreasing on $[0,\pi]$.

The stability of the periodic solution $x(t)$ can now be written in terms of these functions.

\begin{proposition}\label{prop:stability2}
Assume that conditions (i), (ii) and (iii.1) of Theorem~\ref{criterion} hold, and that $x(t)$ is a positive periodic solution of~\eqref{two-term-Abel}. Then,
\begin{equation}\label{eq:I2}
\sgn\left(D'(x_0)\right)=\sgn\left(
\int_{r_1}^{r_2}
\frac{4r
\Big(R(\tau_2(r))-R(\tau_1(r))\Big)
\Big(R(\tau_2(r))+R(\tau_1(r))\Big)}
{\left(r^2-R^2(\tau_2(r))\right)
\left(r^2-R^2(\tau_1(r))\right)}\,dr
\right).
\end{equation}
\end{proposition}

\begin{proof}
By Proposition~\ref{prop:stability1} and taking $t_0=t_1$, we have $\sgn(D'(x_0))=\sgn(I)$, where
\begin{equation*}
I=\int_{t_1}^{t_1+\pi}A(t)\Big(x(t)+x(t+\pi)\Big)^2\,dt.
\end{equation*}
{By~\eqref{eq:difr}},
\begin{equation*}
\begin{split}
I=&\int_{t_{1}}^{t_{1}+\pi}A(t)\left(\frac{1}{y(t)}+\frac{1}{y(t+\pi)}\right)^{2}dt\\
=&\int_{t_{1}}^{t_{2}}A(t)\left(\frac{1}{y(t)}+\frac{1}{y(t+\pi)}\right)^{2}dt+\int_{t_{2}}^{t_{1}+\pi}A(t)\left(\frac{1}{y(t)}+\frac{1}{y(t+\pi)}\right)^{2}dt\\
=&\int_{t_{1}}^{t_{2}}-\frac{r'(t)y(t)y(t+\pi)}{r(t)}\left(\frac{1}{y(t)}+\frac{1}{y(t+\pi)}\right)^{2}dt\\
&+\int_{t_{2}}^{t_{1}+\pi}-\frac{r'(t)y(t)y(t+\pi)}{r(t)}\left(\frac{1}{y(t)}+\frac{1}{y(t+\pi)}\right)^{2}dt.
\end{split}
\end{equation*}
{By the change of variables $r=r(t)$,  and using $r(t)=y(t)+y(t+\pi)$,}
\begin{equation*}
\begin{split}
I&=-\int_{r_{1}}^{r_{2}}\left(\frac{1}{y(\tau_{1}(r))}+\frac{1}{y(\tau_{1}(r)+\pi)}\right)dr-\int_{r_{2}}^{r_{1}}\left(\frac{1}{y(\tau_{2}(r))}+\frac{1}{y(\tau_{2}(r)+\pi)}\right)dr\\
&=\int_{r_{1}}^{r_{2}}\left(\frac{1}{y(\tau_{2}(r))}+\frac{1}{y(\tau_{2}(r)+\pi)}-\frac{1}{y(\tau_{1}(r))}-\frac{1}{y(\tau_{1}(r)+\pi)}\right)\,dr\\
&=\int_{r_{1}}^{r_{2}}\left(\frac{r}{y(\tau_{2}(r))y(\tau_{2}(r)+\pi)}-\frac{r}{y(\tau_{1}(r))y(\tau_{1}(r)+\pi)}\right)\,dr.
    \end{split}
\end{equation*}
Since $r=y(\tau_i(r))+y(\tau_i(r)+\pi)$ for $i=1,2$, we have
\begin{equation*}
4y(\tau_i(r))y(\tau_i(r)+\pi)=r^2-R^2(\tau_i(r)).
\end{equation*}
Therefore,
\begin{equation*}
\begin{split}
I&=\int_{r_{1}}^{r_{2}}\left(\frac{4r}{r^{2}-\left(y(\tau_{2}(r))-y(\tau_{2}(r)+\pi)\right)^{2}}-\frac{4r}{r^{2}-\left(y(\tau_{1}(r))-y(\tau_{1}(r)+\pi)\right)^{2}}\right)dr\\
&=\int_{r_{1}}^{r_{2}}\frac{4r\left(\left(y(\tau_{2}(r))-y(\tau_{2}(r)+\pi)\right)^{2}-\left(y(\tau_{1}(r))-y(\tau_{1}(r)+\pi)\right)^{2}\right)}{\left(r^{2}-\left(y(\tau_{2}(r))-y(\tau_{2}(r)+\pi)\right)^{2}\right)\left(r^{2}-\left(y(\tau_{1}(r))-y(\tau_{1}(r)+\pi)\right)^{2}\right)}\,dr\\
&=
\int_{r_{1}}^{r_{2}}
\frac{4r
\Big(R(\tau_2(r))-R(\tau_1(r)) \Big)
\Big(R(\tau_2(r))+R(\tau_1(r)) \Big)
}{
\left(r^{2}-R^2(\tau_{2}(r))\right)
\left(r^{2}-R^2(\tau_{1}(r)))\right)
}\,dr.
\end{split}
\end{equation*}
\end{proof}

Note that $r=r(\tau_i(r))>|R(\tau_i(r))|$ for $i=1,2$. Therefore, the denominator in~\eqref{eq:I2} is positive. Define
\begin{equation*}
U(r):=R(\tau_2(r))-R(\tau_1(r)),\qquad
V(r):=R(\tau_2(r))+R(\tau_1(r)).
\end{equation*}
Thus, to prove that $D'(x_0)>0$, it is sufficient to show that $U(r)<0$ and $V(r)<0$ for every $r\in(r_1,r_2)$.
Let
\begin{equation*}
K(t):=\frac{A(t)}{B(t)}.
\end{equation*}
Notice that condition (i) of Theorem~\ref{criterion}, together with the continuity of $B$, implies $B(0)=B(\pi)=0$.
Hence $K$ is not defined at $t=\pi$, which corresponds to $r=r_*$.
Condition (i) also gives $K(t+\pi)=-K(t)$ whenever $B(t)\neq0$.
Moreover, under conditions (ii) and (iii.1), $K(t)>0$ for $t\in(0,t_{1})\cup(t_{2},\pi)$ and $K(t)<0$ for $t\in(t_{1},t_{2})$.
See Figure~\ref{fig:K}.

\begin{figure}[t]
\centering
\begin{tikzpicture}[x=1cm,y=0.75cm]
 \draw[->] (0,0)--(12,0) node[right] {$t$};
 \draw[->] (0,-3.3)--(0,3.6) node[above] {$K(t)$};
 \draw[dashed] (9,-3.15)--(9,3.15);
 \draw[blue!65!black,thick]
   plot[domain=16.30655:163.69345,samples=180,variable=\u]
   ({\u/20},{cos(2*\u)/sin(\u)});
 \draw[blue!65!black,thick]
   plot[domain=196.30655:225,samples=70,variable=\u]
   ({\u/20},{cos(2*\u)/sin(\u)});
 \foreach \x in {2.25,6.75,11.25}
 {\draw (\x,0.08)--(\x,-0.08);}
 \node[below left] at (2.25,-0.08) {$t_1$};
 \node[below right] at (6.75,-0.08) {$t_2$};
 \node[below right] at (11.25,-0.08) {$t_1+\pi$};
 \node[below left] at (9,0) {$\pi$};
 \node[below left] at (0,0) {$0$};
\end{tikzpicture}
\caption{Sketch of the function $K(t)$.}
\label{fig:K}
\end{figure}
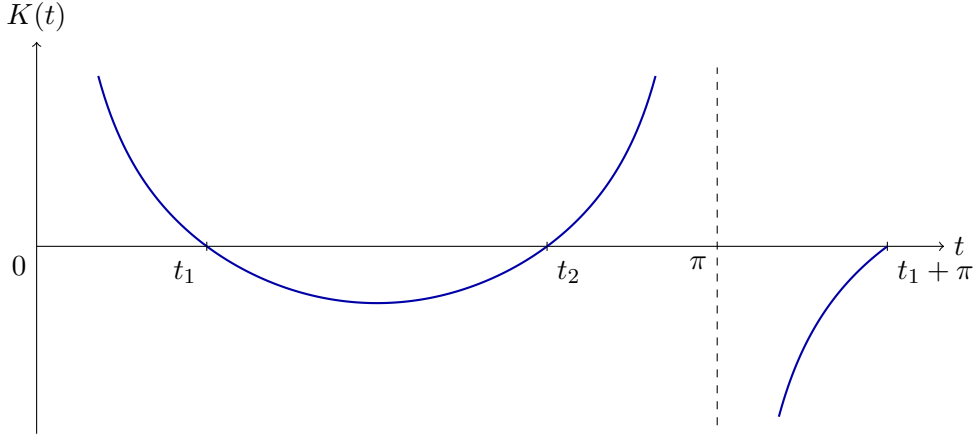
The functions $U$ and $V$ satisfy the following properties.

\begin{proposition}
Assume that conditions (i), (ii) and (iii.1) of Theorem~\ref{criterion} hold, and that $x(t)$ is a positive periodic solution of~\eqref{two-term-Abel}. Then
\begin{enumerate}
\item[(a)] $V(r_1)=0$ and $U(r_2)=0$.

\item[(b)] For $r\in(r_1,r_2)\setminus\{r_*\}$,
\begin{equation}\label{eq:difU}
\frac{dU}{dr}
=
2\frac{y(\tau_2(r)+\pi)y(\tau_2(r))}{rK(\tau_2(r))}
-
2\frac{y(\tau_1(r)+\pi)y(\tau_1(r))}{rK(\tau_1(r))}
-\frac{U(r)}{r}.
\end{equation}

\item[(c)] For $r\in(r_1,r_2)\setminus\{r_*\}$,
\begin{equation}\label{eq:difV}
\frac{dV}{dr}
=
2\frac{y(\tau_2(r)+\pi)y(\tau_2(r))}{rK(\tau_2(r))}
+
2\frac{y(\tau_1(r)+\pi)y(\tau_1(r))}{rK(\tau_1(r))}
-\frac{V(r)}{r}.
\end{equation}
\end{enumerate}
\end{proposition}

\begin{proof}
Since $\tau_1(r_1)=t_1$ and $\tau_2(r_1)=t_1+\pi$,
\begin{equation*}
V(r_{1})=R(\tau_2(r_{1}))+R(\tau_1(r_{1}))
=R(t_1+\pi)+R(t_1)=
-R(t_1)+R(t_1)=0.
\end{equation*} 
Analogously, $\tau_1(r_2)=\tau_2(r_2)=t_2$, so $U(r_2)=0$.

Using~\eqref{eq:difr},
\begin{equation*}
\frac{d\tau_i(r)}{dr}
=
-\frac{y(\tau_i(r))y(\tau_i(r)+\pi)}
{A(\tau_i(r))r},\qquad i=1,2.
\end{equation*}
Differentiating $U$ and $V$ and using the preceding expressions, we obtain
\begin{equation*}
\begin{split}
\frac{dU}{dr} =& 
\frac{dR}{d t}(\tau_2(r)) \frac{d\tau_2}{dr}(r)
- \frac{dR}{d t}(\tau_1(r)) \frac{d\tau_1}{dr}(r) \\
=&
2\frac{B(\tau_{2}(r))}{A(\tau_{2}(r))}\frac{y(\tau_{2}(r)+\pi)y(\tau_{2}(r))}{r}
+\frac{y(\tau_{2}(r)+\pi)-y(\tau_{2}(r))}{r}\\
-&2\frac{B(\tau_{1}(r))}{A(\tau_{1}(r))}\frac{y(\tau_{1}(r)+\pi)y(\tau_{1}(r))}{r}-\frac{y(\tau_{1}(r)+\pi)-y(\tau_{1}(r))}{r}\\
=&
2  \frac{y(\tau_{2}(r)+\pi)y(\tau_{2}(r))}{r K(\tau_{2}(r))}
-2\frac{y(\tau_{1}(r)+\pi)y(\tau_{1}(r))}{r K(\tau_{1}(r))}-\frac{U(r)}{r }.
\end{split}
\end{equation*}
Arguing analogously, we obtain
\begin{equation*}
    \begin{split}
\frac{dV}{dr} =& 
\frac{dR}{d t}(\tau_2(r)) \frac{d\tau_2}{dr}(r)
+ \frac{dR}{d t}(\tau_1(r)) \frac{d\tau_1}{dr}(r) \\
=&
2\frac{B(\tau_{2}(r))}{A(\tau_{2}(r))}\frac{y(\tau_{2}(r)+\pi)y(\tau_{2}(r))}{r}
+\frac{y(\tau_{2}(r)+\pi)-y(\tau_{2}(r))}{r}\\
+&2\frac{B(\tau_{1}(r))}{A(\tau_{1}(r))}\frac{y(\tau_{1}(r)+\pi)y(\tau_{1}(r))}{r}+\frac{y(\tau_{1}(r)+\pi)-y(\tau_{1}(r))}{r}\\
=&
2  \frac{y(\tau_{2}(r)+\pi)y(\tau_{2}(r))}{r K(\tau_{2}(r))}
+2\frac{y(\tau_{1}(r)+\pi)y(\tau_{1}(r))}{r K(\tau_{1}(r))}-\frac{V(r)}{r}.
    \end{split}
\end{equation*}
\end{proof}

We next compute the derivatives of $U$ and $V$ at their zeros.

\begin{proposition}\label{prop:UVzeros}
Let $r_0\in(r_1,r_2)\setminus\{r_*\}$. If $r_0$ is a zero of $V(r)$, then
\begin{equation*}
\frac{dV}{dr}(r_0)
=
\frac{K(\tau_1(r_0))+K(\tau_2(r_0))}
{K(\tau_1(r_0))K(\tau_2(r_0))}
\frac{2y(\tau_2(r_0))y(\tau_1(r_0))}{r_0}.
\end{equation*}
If $r_0$ is a zero of $U(r)$, then
\begin{equation*}
\frac{dU}{dr}(r_0)
=
\frac{K(\tau_1(r_0))-K(\tau_2(r_0))}
{K(\tau_1(r_0))K(\tau_2(r_0))}
\frac{2y(\tau_2(r_0))y(\tau_2(r_0)+\pi)}{r_0}.
\end{equation*}
\end{proposition}

\begin{proof}
If $V(r_0)=0$, then
\begin{equation*}
R(\tau_2(r_0))+R(\tau_1(r_0))=0.
\end{equation*}
Since
\begin{equation*}
r_0=y(\tau_2(r_0))+y(\tau_2(r_0)+\pi)
=y(\tau_1(r_0))+y(\tau_1(r_0)+\pi),
\end{equation*}
we obtain
\begin{equation*}
y(\tau_2(r_0))=y(\tau_1(r_0)+\pi),\qquad
y(\tau_2(r_0)+\pi)=y(\tau_1(r_0)).
\end{equation*}
Thus, by~\eqref{eq:difV},
\begin{equation*}
\frac{dV}{dr}(r_0)
=
\left(
\frac{1}{K(\tau_2(r_0))}
+\frac{1}{K(\tau_1(r_0))}
\right)
\frac{2y(\tau_2(r_0))y(\tau_1(r_0))}{r_0},
\end{equation*}
which gives the first formula.

Analogously, if $U(r_0)=0$, then
\begin{equation*}
y(\tau_2(r_0))=y(\tau_1(r_0)),\qquad
y(\tau_2(r_0)+\pi)=y(\tau_1(r_0)+\pi),
\end{equation*}
and~\eqref{eq:difU} gives
\begin{equation*}
\frac{dU}{dr}(r_0)
=
\left(
\frac{1}{K(\tau_2(r_0))}
-\frac{1}{K(\tau_1(r_0))}
\right)
\frac{2y(\tau_2(r_0))y(\tau_2(r_0)+\pi)}{r_0}.
\end{equation*}
\end{proof}

By Proposition~\ref{prop:UVzeros}, the signs of $U'$ and $V'$ at their zeros are determined by the signs of $K(\tau_1(r))$, $K(\tau_2(r))$, and their sum or difference. The following lemma provides the key information on their sum.

\begin{lemma}\label{lemma2}
Assume that condition (i) of Theorem~\ref{criterion} holds and that $\widehat K$ is strictly convex on $(z(0),z(\pi))$. 
Let $x(t)$ be a positive periodic solution of \eqref{two-term-Abel}, and set $y(t)=1/x(t)$.
Suppose that there exist $0<\tau_1<\tau_2<\pi$ such that
\begin{equation}\label{7592}
y(\tau_1)=y(\tau_2+\pi),\qquad
y(\tau_2)=y(\tau_1+\pi).
\end{equation}
Then $K(\tau_1)+K(\tau_2)>0$.
\end{lemma}

\begin{proof} 
Recall that, by definition, $z(t)=\int_0^t B(s)\,ds$. 
We first prove the following claim. Fix $0<u<v<\pi$. For $t\in(u,v)$, let $t^*(t)\in(u,v)$ be determined by
$z(t)+z(t^*(t))=z(u)+z(v)$.
Then
\begin{equation}\label{eq:K}
K(t)+K(t^*(t))<K(u)+K(v),\qquad u<t<v.
\end{equation}
Indeed, since $B(t)>0$ on $(0,\pi)$, $z(t)$ is strictly increasing. Thus, for $z\in(z(u),z(v))$, strict convexity of $\widehat K$ gives
\begin{equation*}
\widehat K(z)+\widehat K(z(u)+z(v)-z)
<
\widehat K(z(u))+\widehat K(z(v)),
\end{equation*}
which proves~\eqref{eq:K}.

Now take $u=\tau_1$ and $v=\tau_2$. Extend $t^*$ continuously to $[u,v]$ by setting $t^*(u)=v$ and $t^*(v)=u$. Differentiating the identity defining $t^*$ gives
\begin{equation*}
(t^*)'(t)=-\frac{B(t)}{B(t^*(t))}.
\end{equation*}
Set
\begin{equation*}
w(t)=y(t)-y(t^*(t)+\pi),\qquad u\leq t\leq v.
\end{equation*}
By~\eqref{7592}, $w(u)=w(v)=0$. Since $y'(t)=-B(t)-A(t)/y(t)$, we obtain
\begin{equation*}
\begin{split}
\frac{dw(t)}{dt}=&\frac{dy(t)}{dt}-\frac{dy(t^*(t)+\pi)}{dt}\\
            =&-B(t)-\frac{A(t)}{y(t)}+\left(-B(t^{*}(t)+\pi)-\frac{A(t^{*}(t)+\pi)}{y(t^{*}(t)+\pi)}\right)\frac{B(t)}{B(t^{*}(t))}\\
 =&-B(t)-\frac{A(t)}{y(t)}+\left(B(t^{*}(t))-\frac{A(t^{*}(t))}{y(t^{*}(t)+\pi)}\right)\frac{B(t)}{B(t^{*}(t))}\\
  =&-\frac{A(t)}{y(t)}-\frac{A(t^{*}(t))}{y(t^{*}(t)+\pi)}\frac{B(t)}{B(t^{*}(t))}\\
  =&-B(t)\left(\frac{K(t)}{y(t)}+\frac{K(t^{*}(t))}{y(t^{*}(t)+\pi)}\right).
    \end{split}
\end{equation*}
Let
\begin{equation*}
\mu(t)=
\exp\left(
-\int_u^t\frac{A(s)}{y(s)y(t^*(s)+\pi)}\,ds
\right)>0.
\end{equation*}
Then we have
\begin{equation*}
    \begin{split}
\frac{d}{dt}\left(\mu(t) w(t)\right)
=&\frac{d\mu(t)}{dt}w(t)+\mu(t)\frac{dw(t)}{dt}\\
=&\mu(t)\left(-\frac{A(t)}{y(t)y(t^*(t)+\pi)}\right)\left(y(t)-y(t^*(t)+\pi)\right)\\
&-\mu(t)B(t)\left(\frac{K(t)}{y(t)}+\frac{K(t^{*}(t))}{y(t^{*}(t)+\pi)}\right)\\
=&\mu(t)B(t)\left(\frac{K(t)\left(y(t^*(t)+\pi)-y(t)\right)}{y(t)y(t^*(t)+\pi)}\right)\\
&-\mu(t)B(t)\left(\frac{K(t)}{y(t)}+\frac{K(t^{*}(t))}{y(t^{*}(t)+\pi)}\right)\\
=&\mu(t)B(t)\left(\frac{K(t)}{y(t)}-\frac{K(t)}{y(t^*(t)+\pi)}\right)-\mu(t)B(t)\left(\frac{K(t)}{y(t)}+\frac{K(t^{*}(t))}{y(t^{*}(t)+\pi)}\right)\\
=&-\frac{\mu(t)B(t)}{y(t^*(t)+\pi)}
 \left(K(t)+K(t^*(t))\right).
    \end{split}
\end{equation*}
If $K(u)+K(v)\leq0$, then~\eqref{eq:K} gives
\begin{equation*}
\frac{d}{dt}\big(\mu(t)w(t)\big)
>
-\frac{\mu(t)B(t)}{y(t^*(t)+\pi)}
\left(K(u)+K(v)\right)\geq0.
\end{equation*}
This contradicts $w(u)=w(v)=0$. Hence $K(u)+K(v)>0$, and the result follows.
\end{proof}

We are now in a position to prove the main theorem.

\begin{proof}[Proof of Theorem~\ref{criterion}]
We first consider case (iii.1). Let $x_0$ be a positive zero of the displacement map $D$. We shall prove that $D'(x_0)>0$.
By Proposition~\ref{prop:stability2},
\begin{equation*}
\sgn\left(D'(x_0)\right)
=
\sgn\left(
\int_{r_1}^{r_2}
\frac{4rU(r)V(r)}
{\left(r^2-R^2(\tau_2(r))\right)
\left(r^2-R^2(\tau_1(r))\right)}\,dr
\right).
\end{equation*}
Since $r=r(\tau_i(r))>|R(\tau_i(r))|$ for $i=1,2$, the denominator of the integrand is positive. Thus, it suffices to prove that $U(r)<0$ and $V(r)<0$ for every $r\in(r_1,r_2)$.

\medskip

{\bf Claim 1.} $V(r)<0$ for every $r\in(r_1,r_2)$.

\medskip

{\bf Proof of Claim 1.}
First, let $r\in(r_1,r_*]$. Then $0\leq\tau_0(r)<t_1<\tau_1(r)<t_2$ and, by the $\pi$-periodicity of $r$, $\tau_2(r)=\tau_0(r)+\pi$. Since $r(\tau_0(r))=r(\tau_1(r))=r$ and $\tau_0(r)<\tau_1(r)$, the strict decrease of $r(t)R(t)$ on $[0,\pi]$ implied by~\eqref{eq:rRmonotone} gives $R(\tau_0(r))>R(\tau_1(r))$. Hence, using the $\pi$-antiperiodicity of $R$,
\begin{equation*}
V(r)=R(\tau_2(r))+R(\tau_1(r))
=-R(\tau_0(r))+R(\tau_1(r))<0.
\end{equation*}

Now let $r\in(r_*,r_2)$. Then $t_1<\tau_1(r)<t_2<\tau_2(r)<\pi$, so $K(\tau_1(r))<0<K(\tau_2(r))$. Suppose, by contradiction, that $V$ has a zero in $(r_*,r_2)$, and let $r_0$ be its smallest zero. Since $V(r_*)<0$, we have $V'(r_0)\geq0$. On the other hand, if $V(r_0)=0$, the identities in the proof of Proposition~\ref{prop:UVzeros} show that the hypotheses of Lemma~\ref{lemma2} are satisfied by $\tau_1(r_0)$ and $\tau_2(r_0)$. Hence $K(\tau_1(r_0))+K(\tau_2(r_0))>0$, and Proposition~\ref{prop:UVzeros} gives
\begin{equation*}
V'(r_0)
=
\frac{K(\tau_1(r_0))+K(\tau_2(r_0))}
{K(\tau_1(r_0))K(\tau_2(r_0))}
\frac{2y(\tau_2(r_0))y(\tau_1(r_0))}{r_0}<0,
\end{equation*}
a contradiction. Therefore $V(r)<0$ for every $r\in(r_1,r_2)$. This proves Claim~1.

\medskip

{\bf Claim 2.} $U(r)<0$ for every $r\in(r_1,r_2)$.

\medskip

{\bf Proof of Claim 2.}
First, let $r\in[r_*,r_2)$. Then $t_1<\tau_1(r)<t_2<\tau_2(r)\leq\pi$ and $r(\tau_1(r))=r(\tau_2(r))=r$. Since $\tau_1(r)<\tau_2(r)$, the strict decrease of $r(t)R(t)$ implied by~\eqref{eq:rRmonotone} gives $R(\tau_1(r))>R(\tau_2(r))$. Hence $U(r)<0$.

It remains to consider $r\in(r_1,r_*)$. In this case, $0<\tau_0(r)<t_1<\tau_1(r)<t_2$, where $\tau_0(r)=\tau_2(r)-\pi$, and $K(\tau_2(r))=-K(\tau_0(r))$. Suppose, by contradiction, that $U$ has a zero in $(r_1,r_*)$, and let $r_0$ be its largest zero. 
Since $U(r_*)<0$, we have $U'(r_0)\leq0$.
Since $U(r_0)=0$, we have $R(\tau_2(r_0))=R(\tau_1(r_0))$. Using $R(\tau_2(r_0))=-R(\tau_0(r_0))$ and $r(\tau_0(r_0))=r(\tau_1(r_0))=r_0$, we obtain
\begin{equation*}
y(\tau_0(r_0))=y(\tau_1(r_0)+\pi),\qquad
y(\tau_1(r_0))=y(\tau_0(r_0)+\pi).
\end{equation*}
Therefore Lemma~\ref{lemma2}, applied to $\tau_0(r_0)$ and $\tau_1(r_0)$, gives $K(\tau_0(r_0))+K(\tau_1(r_0))>0$. Since $K(\tau_0(r_0))>0$ and $K(\tau_1(r_0))<0$, we also have $K(\tau_2(r_0))=-K(\tau_0(r_0))<0$. Hence Proposition~\ref{prop:UVzeros} gives
\begin{equation*}
\begin{split}
U'(r_0)
&=
\frac{K(\tau_1(r_0))-K(\tau_2(r_0))}
{K(\tau_1(r_0))K(\tau_2(r_0))}
\frac{2y(\tau_2(r_0))y(\tau_2(r_0)+\pi)}{r_0}\\
&=
\frac{K(\tau_1(r_0))+K(\tau_0(r_0))}
{K(\tau_1(r_0))K(\tau_2(r_0))}
\frac{2y(\tau_2(r_0))y(\tau_2(r_0)+\pi)}{r_0}>0,
\end{split}
\end{equation*}
which contradicts $U'(r_0)\leq0$. Thus $U(r)<0$ for every $r\in(r_1,r_2)$. This proves Claim~2.

\vspace{4pt}

By Claims~1 and~2, $U(r)V(r)>0$ for $r\in(r_1,r_2)$. Hence Proposition~\ref{prop:stability2} gives $D'(x_0)>0$, so every positive periodic solution is hyperbolic and unstable.
Therefore equation~\eqref{two-term-Abel} has at most one positive limit cycle.

This proves the theorem under (iii.1). The case (iii.2) follows from the time-reversal transformation considered at the beginning of this section. Time reversal preserves hyperbolicity and reverses stability. Hence, under (iii.2), the positive limit cycle, if it exists, is hyperbolic and stable.
\end{proof}

\section{Application to Rigid Systems}
In this section, we provide the proof of Theorem \ref{mainresult-rigid}. 
As stated in the introduction, studying the limit cycles of equation \eqref{rigid equation1} with $a=0$ can be reduced to studying the positive limit cycles of Abel equation \eqref{Abel1}.
Hence, to prove Theorem \ref{mainresult-rigid}, it is enough to show that Abel equation \eqref{Abel1} has at most one positive limit cycle.
For \eqref{Abel1}, we first reduce the number of parameters by means of a phase shift and a rescaling of the dependent variable.  
We then analyze the parameter region in which \eqref{Abel1} has no limit cycles.
The remaining parameter region will satisfy precisely the half-period sign and convexity conditions of Theorem \ref{criterion}.  
Therefore, there exists at most one limit cycle.
Finally, the sharpness can be shown by an Andronov-Hopf bifurcation occurring at the origin.

Note that when $b=c=0$, equation \eqref{Abel1} reduces to a differential equation with separable variables, so no limit cycles exist in this case. 
Consequently, we shall only consider $b^2+c^2\neq 0$.
The following result provides the normal form of equation \eqref{Abel1}, reducing the original five parameters to three.
\begin{proposition}\label{normalform} 
Assume that $b^2+c^2\neq 0$.
Equation~\eqref{Abel1} is equivalent to
\begin{equation}\label{mainAbel}
\frac{dr}{d\theta} =  \sin \theta\, r^2 + \left(\alpha + \beta \cos(2 \theta)+\gamma\sin(2 \theta) \right)r^3,
\end{equation}
where
$$
\alpha = \frac{d+f}{2(b^2+c^2)},\quad \beta = \frac{(d-f)(c^{2}-b^{2})-2bce}{2(b^2+c^2)^{2}},\quad \gamma = \frac{2bc(d-f)+e(c^{2}-b^{2})}{2(b^2+c^2)^{2}}.
$$
\end{proposition}
\begin{proof}
    Let $\theta_0\in[0,2\pi]$ be such that $\sin \theta_0 = b/\sqrt{b^2+c^2}$, $\cos \theta_0 = c/\sqrt{b^2+c^2}$.
Then the change of variables $\theta\mapsto \theta-\theta_0$ transforms \eqref{Abel1} into
\begin{equation*}
\frac{dr}{d\theta} = \sqrt{b^2+c^2} \sin \theta\ r^2 + \left(d\cos^2 (\theta-\theta_{0})+e\cos (\theta-\theta_{0})\sin (\theta-\theta_{0})+f\sin^2 (\theta-\theta_{0})\right)r^3.
\end{equation*}
Write the coefficients of $r^{3}$ in double-angle form and we have
\begin{equation*}
\frac{dr}{d\theta} = \sqrt{b^2+c^2} \sin \theta\ r^2 + \left(\hat{\alpha}+\hat{\beta}\cos(2\theta)+\hat{\gamma}\sin(2\theta)\right)r^3,
\end{equation*}
where 
$$\hat{\alpha}=\frac{d+f}{2},\quad \hat{\beta}=\frac{(d-f)(c^{2}-b^{2})-2bce}{2(b^2+c^2)},\quad \hat{\gamma}=\frac{2bc(d-f)+e(c^{2}-b^{2})}{2(b^2+c^2)}.$$
Finally, the rescaling $r\mapsto\sqrt{b^2+c^2}\,r$ gives \eqref{mainAbel}.
\end{proof}

Next, we consider equation \eqref{mainAbel}.
The change $\theta\mapsto -\theta$ transforms \eqref{mainAbel} into 
$$\frac{dr}{d\theta} =  \sin(\theta) r^2 + \left(-\alpha - \beta \cos(2 \theta)+\gamma\sin(2 \theta) \right)r^3.$$ 
Thus, without loss of generality, we may assume $\alpha+\beta\geq0$.
Moreover, we define 
\begin{equation*}
    A(\theta):=\alpha + \beta \cos(2 \theta)+\gamma\sin(2 \theta)\quad \text{and}\quad B(\theta):=\sin\theta.
\end{equation*}
If $\alpha^2\geq\beta^2+\gamma^2$, the coefficient
$A$ does not change sign, and by classical criterion \cite{Pliss,GL}, it follows that the equation has at most one limit cycle.
Therefore, we assume that $\alpha^2<\beta^2+\gamma^2$, that is, $A$ changes sign on $(0,\pi)$.

To identify the parameter region in which a limit cycle may occur, 
we first study the stability of $r=0$.
In addition to determining the center condition, the next lemma will be used at the end of the proof to show that the upper bound in Theorem \ref{mainresult-rigid} is sharp.
\begin{lemma}\label{lyapunov-center}
For equation \eqref{mainAbel}, the first two Lyapunov quantities of $r=0$ are $\alpha$ and $\beta$, respectively.
Furthermore, $r=0$ is a center if and only if $\alpha=\beta=0$.
\end{lemma}
\begin{proof}
Let $D(r_{0})=r(2\pi;r_{0})-r_{0}$ be the displacement map of \eqref{mainAbel}. 
For sufficiently small $r_0$, analytic dependence on the initial value allows us to write the solution with $r(0;r_0)=r_0$ as
\begin{equation*}
r(\theta;r_0)=r_0+\sum_{j=2}^{5}h_j(\theta)r_0^j+O(r_0^6),
\qquad h_j(0)=0,
\end{equation*}
uniformly for $\theta\in[0,2\pi]$.
Substituting this expansion into \eqref{mainAbel} and comparing coefficients of $r_0^j$ gives
\begin{equation*}
\begin{aligned}
\frac{dh_2}{d\theta}&=B,\\
\frac{dh_3}{d\theta}&=2Bh_2+A,\\
\frac{dh_4}{d\theta}&=B(2h_3+h_2^2)+3Ah_2,\\
\frac{dh_5}{d\theta}&=B(2h_4+2h_2h_3)+A(3h_3+3h_2^2).
\end{aligned}
\end{equation*}
The first two equations give
\begin{equation*}
h_2(\theta)=1-\cos\theta,\qquad
h_3(\theta)=(1-\cos\theta)^2+\int_0^\theta A(s)\,ds.
\end{equation*}
In particular, $h_2(0)=h_2(2\pi)=0$.
Integrating the recurrence over $[0,2\pi]$ and using integration by parts, we obtain
\begin{equation*}
\begin{aligned}
h_3(2\pi)&=\int_0^{2\pi}A(\theta)\,d\theta=2\pi\alpha,\\
h_4(2\pi)&=\int_0^{2\pi}A(\theta)(1-\cos\theta)\,d\theta=2\pi\alpha,\\
h_5(2\pi)&=\frac32h_3(2\pi)^2
+\int_0^{2\pi}A(\theta)(1-\cos\theta)^2\,d\theta\\
&=6\pi^2\alpha^2+3\pi\alpha+\frac{\pi}{2}\beta,
\end{aligned}
\end{equation*}
where the integrals are evaluated by trigonometric orthogonality, using
$(1-\cos\theta)^2=\frac32-2\cos\theta+\frac12\cos(2\theta)$.
Since $D(r)=\sum_{j=2}^{5}h_j(2\pi)r^j+O(r^6)$, it follows that
\begin{equation}\label{displacement-mainAbel}
 D(r)=2\pi\alpha r^3+2\pi\alpha r^4+
 \left(6\pi^2\alpha^2+3\pi\alpha+\frac{\pi}{2}\beta\right)r^5
 +O(r^6).
\end{equation}
The first nonzero coefficient is $2\pi\alpha$ when $\alpha\neq0$.
For $\alpha=0$, the next possible nonzero coefficient is $\frac{\pi}{2}\beta$.
Thus, after normalization by positive constants, the first two Lyapunov quantities are $\alpha$ and $\beta$.
In particular, a necessary condition for $r=0$ to be a center is $\alpha=\beta=0$.
Conversely, if $\alpha=\beta=0$, both coefficients $\sin\theta$ and $\gamma\sin(2\theta)$ change sign under the reflection $\theta\mapsto2\pi-\theta$.
Hence the flow on $[\pi,2\pi]$ is the inverse of the flow on $[0,\pi]$, so $D(r)=0$ for all sufficiently small $r$ and $r=0$ is a center.
\end{proof}

Next, we study the parameter region for which equation \eqref{mainAbel} has no limit cycles.
We first recall the following nonexistence criterion for limit cycles introduced by Bravo and Torregrosa \cite{BT}.
\begin{lemma}[{\cite[Theorem~2.8]{BT}}]\label{lemma:BT-nonexistence}
Consider
\begin{equation*}
\frac{dr}{d\theta}=(f(\theta)+g(\theta))r^2+h(\theta)r^3,
\end{equation*}
where $f,g,h$ are continuous and $2\pi$-periodic, with $f\not\equiv0$ and $h\not\equiv0$.
Assume that the following conditions hold:
\begin{enumerate}
\item[(i)] $f(\pi-\theta)=-f(\theta)$, and $f$ has a fixed sign on $(\pi/2,3\pi/2)$;
\item[(ii)] $g(-\theta)=-g(\theta)$, and $g(\theta)+g(\pi-\theta)$ has a fixed sign on $(0,\pi/2)$ and is not identically zero there;
\item[(iii)] $h(-\theta)=-h(\theta)$ and $h(\pi-\theta)=-h(\theta)$, and $h$ has a fixed sign on $(0,\pi/2)$.
\end{enumerate}
Then the equation admits no positive periodic solution.
\end{lemma}

Applying the above nonexistence criterion for limit cycles together with the comparison principle, we obtain the following result.
\begin{proposition}\label{nonexistence-mainAbel}
Equation \eqref{mainAbel} has no positive limit cycles if
$\alpha\beta\geq0$ or $|\alpha|\geq|\beta|$.
\end{proposition}
\begin{proof}
For $\alpha=\beta=0$, by Lemma \ref{lyapunov-center} there is no limit cycles.
For $\alpha=0$ and $\beta\neq0$, choose $\delta$ such that
\begin{equation*}
\sin(2\delta)=-\frac{\beta}{\sqrt{\beta^2+\gamma^2}},
\qquad
\cos(2\delta)=\frac{\gamma}{\sqrt{\beta^2+\gamma^2}}.
\end{equation*}
The map $\theta\mapsto\theta+\delta$ puts \eqref{mainAbel} in the form of Lemma~\ref{lemma:BT-nonexistence}, with
\begin{equation*}
f(\theta)=\sin\delta\cos\theta,\qquad
g(\theta)=\cos\delta\sin\theta,\qquad
h(\theta)=\sqrt{\beta^2+\gamma^2}\sin(2\theta).
\end{equation*}
Since $\sin\delta\cos\delta\neq0$ and $g(\theta)+g(\pi-\theta)=2\cos\delta\sin\theta$, all the hypotheses of Lemma \ref{lemma:BT-nonexistence} hold.
So, equation \eqref{mainAbel} has no positive periodic solution. 
Hence the displacement function has no zero in the positive domain.
Since this domain is an interval, the displacement has a constant
sign there. Formula \eqref{displacement-mainAbel} shows that this
sign is precisely the sign of $\beta$.
Furthermore, by the comparison principle, it is easy to see that when $\alpha\beta\geq0$, the equation has no limit cycles.

Finally, set $\widetilde A(\theta)=-A(2\pi-\theta)$.
The reflection $\theta\mapsto2\pi-\theta$ transforms \eqref{mainAbel} into
\begin{equation}\label{8964}
\frac{dr}{d\theta}
=\widetilde A(\theta)r^3+B(\theta)\,r^2
=\bigl(-\alpha-\beta\cos(2\theta)+\gamma\sin(2\theta)\bigr)r^3+\sin\theta\,r^2.
\end{equation}
If $|\alpha|\geq|\beta|$ and $(\alpha,\beta)\neq(0,0)$, then
\begin{equation*}
A(\theta)-\widetilde A(\theta)
=A(\theta)+A(2\pi-\theta)=2\alpha+2\beta\cos(2\theta)
\end{equation*}
has a fixed sign and is not identically zero on $[0,2\pi]$.
Suppose that $r(\theta)$ is a positive periodic solution of \eqref{mainAbel}.
Then $\widetilde r(\theta):=r(2\pi-\theta)$ is a positive periodic solution of the reflected equation \eqref{8964}, and periodicity gives
$\widetilde r(0)=r(0)$ and $\widetilde r(2\pi)=r(2\pi)$.
However, the two right-hand sides differ by $(A(\theta)-\widetilde A(\theta))r^3$.
Since both solutions are positive and have the same initial value, strict comparison yields
$r(2\pi)>\widetilde r(2\pi)$ if $\alpha>0$, and
$r(2\pi)<\widetilde r(2\pi)$ if $\alpha<0$.
This contradicts the endpoint identities above, so no positive periodic solution exists.
\end{proof}

{Summarizing, our space of parameters is $\alpha+\beta>0$ and it is already known that~\eqref{8964} has at most one positive limit cycle when $\alpha^2<\beta^2+\gamma^2$, and it has no limit cycles when $\alpha\beta \leq 0$ or $|\alpha|\leq |\beta|$. Note that $|\alpha|\leq |\beta|$ implies $\alpha^2<\beta^2\leq\beta^2+\gamma^2$.} 
Therefore,
to study the uniqueness of limit cycles we only need to consider the following parameter region~
(see Figure~\ref{fig:alpha-beta-regions}):
\begin{itemize}
    \item[(H)] $\alpha<0$, and {$\beta>-\alpha$}.
\end{itemize}
\begin{figure}[t]
\centering
\begin{tikzpicture}[x=1cm,y=1cm]
  \fill[pattern=north east lines,pattern color=black!45]
    (0,0) rectangle (3.2,3.2);
  \fill[pattern=north east lines,pattern color=black!45]
    (0,0)--(3.2,0)--(3.2,-3.2)--cycle;
  \draw[thick] (-3.35,3.35)--(3.35,-3.35);
  \node[above right] at (-3.2,3.2) {$\beta=-\alpha$};
  \draw[->,thick] (-3.65,0)--(3.7,0) node[right] {$\alpha$};
  \draw[->,thick] (0,-3.65)--(0,3.7) node[above] {$\beta$};
  \node[fill=white,inner sep=1pt,below right] at (0,0) {$0$};
  \node[font=\large] at (-1.1,2.2) {$(H)$};
\end{tikzpicture}
\caption{Parameter regions for \eqref{mainAbel} in the half-plane $\alpha+\beta\geq0$. 
In the hatched region, \eqref{mainAbel} has no positive limit cycles.
In the region $(H)$, \eqref{mainAbel} has at most one positive limit cycle.}
\label{fig:alpha-beta-regions}
\end{figure}
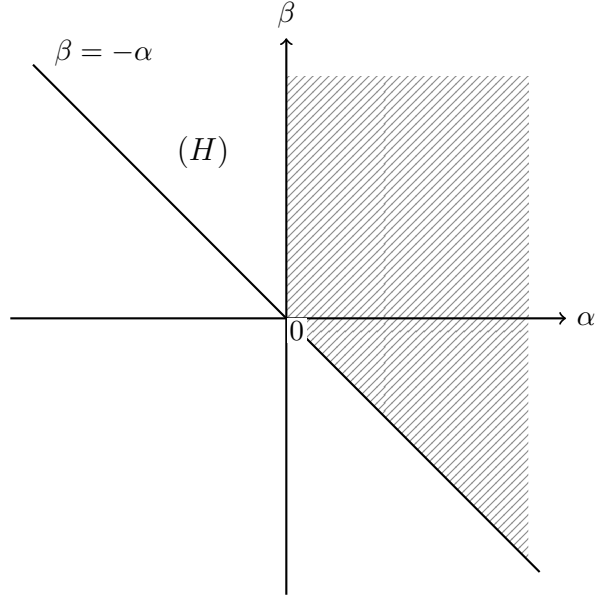

By applying Theorem \ref{criterion}, we have the following result.
\begin{proposition}\label{at most one for H}
Under condition $(H)$, equation \eqref{mainAbel} has at most one positive limit cycle.
Moreover, if it exists, it is hyperbolic and unstable.
\end{proposition}
\begin{proof}
Recall that for equation \eqref{mainAbel}, the coefficients are
\begin{equation*}
A(\theta)=\alpha+\beta\cos(2\theta)+\gamma\sin(2\theta)\ \ \text{ and }\ \ B(\theta)=\sin\theta.
\end{equation*}
They are smooth and $2\pi$-periodic. We verify the three conditions of Theorem~\ref{criterion} in order.
\begin{enumerate}
\item[(i)] The trigonometric identities give
$A(\theta+\pi)=A(\theta)$, $B(\theta+\pi)=-B(\theta)$ and $B(\theta)>0$ for $0<\theta<\pi$.
Thus condition \textup{(i)} holds.

\item[(ii)] By $(H)$, $\alpha^2<\beta^2+\gamma^2$, so $A$ has exactly two zeros in a period.
Since $A(0)=A(\pi)=\alpha+\beta>0$, both zeros lie in $(0,\pi)$.
Denote them by $\theta_1<\theta_2$.
Moreover, both zeros are simple, and condition \textup{(ii)} holds.

\item[(iii)] The positivity of $A(0)$ and $A(\pi)$ and the simplicity of the two zeros imply
\begin{equation*}
A<0\quad\text{on }(\theta_1,\theta_2),\qquad
A>0\quad\text{on }(0,\theta_1)\cup(\theta_2,\pi).
\end{equation*}
It therefore suffices to verify the strict convexity required in \textup{(iii.1)}.
The change of independent variable in Theorem~\ref{criterion} is
\begin{equation*}
z=\int_0^\theta B(s)\,ds=1-\cos\theta,\qquad 0<z<2.
\end{equation*}
Using $\cos\theta=1-z$ and $\sin\theta=\sqrt{z(2-z)}$, we obtain
\begin{equation*}
\widehat K(z)=\frac{A(\theta(z))}{B(\theta(z))}
=\frac{\alpha-\beta+2\beta(1-z)^2}{\sqrt{z(2-z)}}+2\gamma(1-z).
\end{equation*}
Differentiating twice gives
\begin{equation*}
\frac{d^2\widehat K}{dz^2}
=\frac{\alpha+3\beta+2\alpha(1-z)^2}{[z(2-z)]^{5/2}}
\geq\frac{3\left(\alpha+\beta\right)}{[z(2-z)]^{5/2}}>0
, \qquad 0<z<2.
\end{equation*}
Thus $\widehat K$ is strictly convex on $(0,2)$, and condition \textup{(iii.1)} holds.
\end{enumerate}
By Theorem~\ref{criterion} we know that \eqref{mainAbel} has at most one positive limit cycle.
If such limit cycle exists, it is hyperbolic and unstable.
\end{proof}

Now we are ready to give the proof of Theorem \ref{mainresult-rigid}.
\begin{proof}[Proof of Theorem \ref{mainresult-rigid}]
By Lemma \ref{lyapunov-center}, Proposition \ref{nonexistence-mainAbel} and Proposition \ref{at most one for H}, 
we know that equation \eqref{mainAbel} has at most one limit cycle.
After the normalization in Proposition \ref{normalform} and, if necessary, the change $\theta\mapsto-\theta$, we may assume $\alpha+\beta\ge0$. 
Outside region (H), Proposition \eqref{nonexistence-mainAbel} shows that no positive limit cycle exists, whereas in region (H), Proposition \ref{at most one for H} gives at most one. 
Hence equation \eqref{mainAbel} has at most one positive limit cycle.

Next, we show that the bound is sharp.  
Fix $\beta>0$ and $\gamma\in\mathbb R$.
Formula \eqref{displacement-mainAbel} shows that the weak focus at the origin
has second Lyapunov quantity of sign $\beta$ when $\alpha=0$.  Hence, for
$\alpha<0$ sufficiently close to zero, a small positive hyperbolic periodic
solution bifurcates from the origin.  Therefore the upper bound is sharp.

Finally, by Proposition \ref{normalform} and the correspondence between the limit cycles of system \eqref{rigid equation1} and positive limit cycles of equation \eqref{Abel1}, 
the conclusion follows.
\end{proof}

\section{Application to Quadratic Systems}
In this section, we apply Theorem~\ref{criterion} to the quadratic planar system \eqref{quadratic-system}.
For system~\eqref{quadratic-system}, set
\begin{equation*}
\begin{split}
h(\theta)&=\cos\theta\,P_2(\cos\theta,\sin\theta)
           +\sin\theta\,Q_2(\cos\theta,\sin\theta),\\
g(\theta)&=\cos\theta\,Q_2(\cos\theta,\sin\theta)
           -\sin\theta\,P_2(\cos\theta,\sin\theta).
\end{split}
\end{equation*}
Using a polar coordinate change and Cherka's change of variable, $x=r\cos\theta$, $y=r\sin\theta$ and $\rho=r/(1+r g(\theta))$, the number of limit cycles of~\eqref{quadratic-system} surrounding the origin is the number of positive limit cycles of the Abel equation (see, e.g., \cite{Lins})
\begin{equation*}
\rho' = A(\theta) \rho^3+B(\theta)\rho^2,
\end{equation*}
where 
\begin{equation*}
A(\theta) =-h(\theta) g(\theta),\quad B(\theta)=h(\theta)-g'(\theta).
\end{equation*}

\subsection{Uniqueness criterion for quadratic systems}

As $B$ is a homogeneous trigonometric polynomial, it satisfies $ B(\theta+\pi)=-B(\theta)$.
In particular, there exists $\theta_0$ such that $B(\theta_0)=0$. Hence, after a translation of the angular variable $\theta$, there is no loss of generality in assuming that
\begin{equation*}
    B(0)=0,\qquad B(\theta)>0 \quad \text{for sufficiently small $\theta>0$ }. 
\end{equation*}

We now prove Theorem~\ref{quadratic-criterion}. 
\begin{proof}[Proof of Theorem~\ref{quadratic-criterion}]
Note that $A,B$ are homogeneous trigonometric polynomials of degrees $6$ and $3$, respectively. In particular, they are smooth, $2\pi$-periodic, and they satisfy
$A(\theta+\pi)=A(\theta)$,
$B(\theta+\pi)=-B(\theta)$.

By assumption (i), $B(\theta)>0$ for $\theta\in(0,\pi)$.
Hence condition (i) of Theorem~\ref{criterion} is satisfied.

\medskip

By assumption (ii), $A$ has exactly two simple zeros
$0<\theta_1<\theta_2<\pi$.
Thus condition (ii) of Theorem~\ref{criterion} also holds.

\medskip

It remains to verify condition (iii) of Theorem~\ref{criterion}. 
Set
\begin{equation*}
z(\theta)=\int_0^\theta B(s)\,ds,
\qquad
K(\theta)=\frac{A(\theta)}{B(\theta)},
\qquad 
\widehat K(z)=K(\theta(z)).
\end{equation*}
By assumption~(i), $z$ is strictly increasing. The identity
$d/dz=B^{-1}d/d\theta$ gives
\begin{equation*}
\frac{d^2\widehat K}{dz^2}
=\frac{B^2A''-ABB''-3BA'B'+3A(B')^2}{B^5}
=\frac{\mathcal C}{B^5}.
\end{equation*}
If $A(0)>0$, then $A>0$ on $(0,\theta_1)\cup(\theta_2,\pi)$,
$A<0$ on $(\theta_1,\theta_2)$, and $\mathcal C>0$.
Hence alternative~(iii.1) of Theorem~\ref{criterion} applies.
If $A(0)<0$, all these signs are reversed and
alternative~(iii.2) applies.
The conclusion follows from Theorem~\ref{criterion} and the preceding
correspondence.
\end{proof}

{Next, we will rewrite conditions of Theorem~\ref{quadratic-criterion} in terms of the coefficients. We show that they can be reduced to a semialgebraic system in terms of the parameters of the system.}
Write
\begin{equation}\label{jutiquadratic}
\begin{cases}
&\dot{x}=P_2(x,y)=p_0x^2+p_1xy+p_2y^2,\\
&\dot{v}=Q_2(x,y)=q_0x^2+q_1xy+q_2y^2.
\end{cases}
\end{equation}
By direct computation, we have
\begin{equation*}
\begin{split}
    B(\theta)=&(2p_0-q_1)\cos^3 \theta+(3p_1+4q_0-2q_2)\cos^2 \theta\sin \theta\\
&+(4p_2-2p_0+3q_1)\cos \theta\sin^2 \theta+(2q_2-p_1)\sin^3 \theta.
\end{split}
\end{equation*}
Choosing a zero of $B$ as the angular origin gives $q_1=2p_0$.
Moreover, $B(\frac{\pi}{2})=2q_{2}-p_{1}$, so rotating by $\pi$ if necessary, we may assume that $p_1<2q_{2}$. 
With $u=\cot\theta$, $0<\theta<\pi$, define
\begin{equation}\label{eq:quadratic-polynomials}
\begin{split}
P_B(u)&=(3p_1+4q_0-2q_2)u^2+4(p_0+p_2)u+2q_2-p_1,\\
H(u)&=p_0u^3+(p_1+q_0)u^2+(p_2+2p_0)u+q_2,\\
G(u)&=q_0u^3+p_0u^2+(q_2-p_1)u-p_2.
\end{split}
\end{equation}
Then
\begin{equation}\label{eq:quadratic-rational}
B=\frac{P_B(u)}{(1+u^2)^{3/2}},\qquad
h=\frac{H(u)}{(1+u^2)^{3/2}},\qquad
g=\frac{G(u)}{(1+u^2)^{3/2}}.
\end{equation}
Moreover, if the leading coefficients of the polynomials in \eqref{eq:quadratic-polynomials} are nonzero, 
their discriminants can be written as
\begin{equation*}
    \begin{split}
        &\Delta_{P_B} = 16(p_0+p_2)^2 - 4(3p_1+4q_0-2q_2)(2q_2-p_1),\\
        &\Delta_H = (p_1 + q_0)^2 (p_2 + 2p_0)^2 - 4p_0 (p_2 + 2p_0)^3 - 4(p_1 + q_0)^3 q_2 - 27p_0^2 q_2^2 + 18p_0 (p_1 + q_0) (p_2 + 2p_0) q_2,\\
        &\Delta_G = p_0^2 (q_2 - p_1)^2 - 4q_0 (q_2 - p_1)^3 + 4p_0^3 p_2 - 27q_0^2 p_2^2 - 18q_0 p_0 (q_2 - p_1) p_2.
    \end{split}
\end{equation*}
The next proposition gives coefficient tests without expanding the
full convexity polynomial.

\begin{proposition}\label{prop:quadratic}
    Assume that $q_1=2p_0$, $2q_{2}-p_1>0$ and the following conditions hold:
    \begin{enumerate}
    \item[(i)] $\Delta_{P_B}<0$;
    \item[(ii)] $p_0q_0\neq0$ and $\Delta_H<0$, $\Delta_G<0$;
    \item[(iii)] $p_0q_0(P_BL'-3P_B'L)>0$ for every $u\in\mathbb R$, where $L(u)=(1+u^2)\bigl(P_B(HG)'-P_B'HG\bigr)-3uP_BHG$.
    \end{enumerate}
    Then system \eqref{jutiquadratic} has at most one limit cycle surrounding the origin. 
\end{proposition}
\begin{proof}
By \eqref{eq:quadratic-rational}, the sign of $B(\theta)$ on
$(0,\pi)$ is the same as that of $P_B(u)$ on $\mathbb R$.
Since $P_B(0)=2q_2-p_1>0$
and $\Delta_{P_B}<0$, the polynomial $P_B$ has no real zeros.
Hence $B(\theta)>0$ for $\theta\in(0,\pi)$.

Since $p_0q_0\neq0$, both $H$ and $G$ are cubic polynomials.
Moreover, $\Delta_H<0$ and $\Delta_G<0$ imply that each of them
has exactly one simple real zero. Let these zeros be $u_H$ and
$u_G$, respectively.
Since $u=\cot\theta$
maps $(0,\pi)$ bijectively onto $\mathbb R$, and
\begin{equation*}
A(\theta)=-h(\theta)g(\theta)
=-\frac{H(u)G(u)}{(1+u^2)^3},
\end{equation*}
if $u_H\neq u_G$, then $A(\theta)$ has exactly two simple zeros in $(0,\pi)$, and
if $u_H = u_G$, then $A(\theta)$ has a constant sign on $(0,\pi)$ and is not identically zero.
It is already known from \cite{BravoEtAl2019} that if $A(\theta)$ has a constant sign  system~\eqref{jutiquadratic} has no limit cycles.

It remains to verify the convexity condition in
Theorem~\ref{quadratic-criterion}. 
Using $\frac{du}{d\theta}=-(1+u^2)$
and \eqref{eq:quadratic-rational}, a direct computation gives
\begin{equation*}
\mathcal C(\theta)
=
-\frac{P_B(u)L'(u)-3P_B'(u)L(u)}
{(1+u^2)^5}.
\end{equation*}
On the other hand,
$A(0)=-h(0)g(0)=-p_0q_0$.
Therefore,
\begin{equation*}
A(0)\mathcal C(\theta)
=
\frac{
p_0q_0\bigl(P_B(u)L'(u)-3P_B'(u)L(u)\bigr)
}
{(1+u^2)^5}.
\end{equation*}
By assumption~(iii),
$A(0)\mathcal C(\theta)>0$ for $\theta\in(0,\pi)$.
Thus all the hypotheses of Theorem~\ref{quadratic-criterion} are
satisfied, and system~\eqref{jutiquadratic} has at most one limit
cycle surrounding the origin.

The conclusion follows.
\end{proof}

\subsection{The two-parameter family}
In this subsection, we prove corollary~\ref{mainresult-quadratic} for the two-parameter family of quadratic system~\eqref{explicit-quadratic}. Recall that the number of limit cycles of~\eqref{explicit-quadratic} that surround the origin is the number of positive limit cycles for the corresponding Abel equation. Moreover, it could be proved (though as it is a long technical proof not related with the objectives of the paper it is not included) that under the hypotheses of corollary~\ref{mainresult-quadratic}, the system~\eqref{explicit-quadratic} has exactly two critical points, the origin, which is of center type and another one that is a saddle. As it is known that each limit cycle of a quadratic system surrounds exactly one critical point, which must be of focus type (see, e.g., \cite{Coppel1966}), system~\eqref{explicit-quadratic} with $|a|\leq 1/3$ and $1\leq b\leq 3/2$ has exactly one limit cycle.

\medskip 

With the notation of previous subsection, we have
\[
p_0=a,\quad p_1=2b,\quad p_2 =a,\quad 
q_0=1,\quad q_1=2a,\quad q_2=3b.
\]
In particular, $q_1=2p_0$, and $2q_2>p_1$ for $b>0$, so we are in the hypotheses of Proposition~\eqref{prop:quadratic}. 
Then the polynomials in
\eqref{eq:quadratic-polynomials} are
\begin{equation*}
\begin{split}
H(u)&=au^3+(2b+1)u^2+3au+3b,\\
G(u)&=u^3+au^2+bu-a,\\
P_B(u)&=4(u^2+2au+b).
\end{split}
\end{equation*}
We first give the positivity estimate used to verify convexity.

\begin{lemma}\label{lem:family-positivity}
For $|a|\leq1/3$ and $1\leq b\leq3/2$, define
\begin{equation*}
\begin{split}
R_{a,b}(u)={}&3u^8+20au^7+(42a^2+13b-1)u^6+4a(4a^2+16b-1)u^5\\
&+\bigl((26b+8)a^2+29b^2-11b\bigr)u^4+4a(2b^2+6b-1)u^3\\
&+\bigl((8b-2)a^2-5b^3+37b^2-20b\bigr)u^2\\
&+4ab(5b-8)u+b(4b-b^2-18a^2).
\end{split}
\end{equation*}
Then $R_{a,b}(u)>0$ for every $u\in\mathbb R$.
\end{lemma}

\begin{proof}
Since $a^2\leq 1/9$ and $1\leq b\leq 3/2$, we have
\begin{equation*}
\frac{26}{3}a^2+13b-1\geq12b,
\end{equation*}
\begin{equation*}
(26b+8)a^2+29b^2-11b\geq18b^2,
\end{equation*}
and
\begin{equation*}
(8b-2)a^2-5b^3+37b^2-20b
\geq\frac{19}{2}b^2.
\end{equation*}
Hence, completing a square in the three highest-degree terms gives
\begin{align*}
R_{a,b}(u)\geq{}&
3u^6\left(u+\frac{10a}{3}\right)^2\\
&+u^4\left[
12bu^2+4a(4a^2+16b-1)u+11b^2
\right]\\
&+u^2\left[
7b^2u^2+4a(2b^2+6b-1)u+\frac{19}{2}b^2
\right]\\
&+4ab(5b-8)u+b(4b-b^2-18a^2).
\end{align*}
We next show that the two quadratic expressions in brackets are
positive. Since
$4a^2+16b-1<16b$,
the discriminant of the first one satisfies
\begin{align*}
16a^2(4a^2+16b-1)^2-528b^3<
\frac{4096}{9}b^2-528b^3
\leq-\frac{656}{9}b^2<0.
\end{align*}
Therefore,
$12bu^2+4a(4a^2+16b-1)u+11b^2>0$.
For the second one, we claim that
\begin{equation*}
7b^2u^2+4a(2b^2+6b-1)u+\frac{19}{2}b^2>5.
\end{equation*}
Indeed, since $2b^2+6b-1\leq8b^2$,
the discriminant of $7b^2u^2+4a(2b^2+6b-1)u+\frac{19}{2}b^2-5$
satisfies
\begin{align*}
&16a^2(2b^2+6b-1)^2
-28b^2\left(\frac{19}{2}b^2-5\right)\\
&\leq
\frac{1024}{9}b^4-266b^4+140b^2=
b^2\left(140-\frac{1370}{9}b^2\right)
\leq-\frac{110}{9}b^2<0.
\end{align*}
Thus,
$$
R_{a,b}(u)>
5u^2+4ab(5b-8)u+b(4b-b^2-18a^2).$$
Finally, $
b(4b-b^2-18a^2)\geq1$,
and, since $1\leq b\leq3/2$, $b(8-5b)\leq3$.
It follows that the discriminant of the quadratic polynomial on the
right-hand side satisfies
\begin{align*}
16a^2b^2(5b-8)^2
-20b(4b-b^2-18a^2)\leq16-20=-4<0.
\end{align*}
Therefore,
\begin{equation*}
5u^2+4ab(5b-8)u+b(4b-b^2-18a^2)>0
\end{equation*}
for every $u\in\mathbb R$. Consequently,
$
R_{a,b}(u)>0$ for $u\in\mathbb R$.
\end{proof}


Now, we prove Corollary~\ref{mainresult-quadratic}.
\begin{proof}[Proof of Corollary~\ref{mainresult-quadratic}]

We first consider $a=0$. In this case system~\eqref{explicit-quadratic}
is reversible with respect to the $x$-axis, that is, it is invariant
under $(x,y,t)\longmapsto (x,-y,-t)$.
Since the linearization at the origin has eigenvalues $\pm i$, the
origin is a center. Thus, it remains to consider $a\neq0$.

We verify the hypotheses of {Proposition~\ref{prop:quadratic}}. 

\begin{enumerate}
\item[(i)] First,
\begin{equation*}
\Delta_{P_B}=64(a^2-b)<0,
\end{equation*}
because $a^2\leq1/9$ and $b\geq1$.

\item[(ii)] Next, $p_0q_0=a\neq0$. For $H$, a direct computation gives
\begin{equation*}
\Delta_H
=
-108a^4+(117b^2+198b+9)a^2
-12b(2b+1)^3.
\end{equation*}
Since $a^2\leq1/9$ and $b\geq1$,
$
(117b^2+198b+9)a^2
\leq13b^2+22b+1
\leq36b^2.
$
Therefore,
\begin{equation*}
\Delta_H
\leq
36b^2-12b(2b+1)^3<0.
\end{equation*}
On the other hand,
\begin{equation*}
G'(u)
=
3u^2+2au+b
=
3\left(u+\frac{a}{3}\right)^2
+b-\frac{a^2}{3}>0.
\end{equation*}
Hence $G$ has exactly one simple real zero, and consequently
$\Delta_G<0$.

\item[(iii)] Finally, a direct computation gives
\begin{equation*}
P_B(u)L'(u)-3P_B'(u)L(u)
=
16a(1+u^2)R_{a,b}(u).
\end{equation*}
By the Lemma~\ref{lem:family-positivity},
$R_{a,b}(u)>0$ for every $u\in\mathbb R$.
It follows that
\begin{equation*}
p_0q_0\bigl(P_BL'-3P_B'L\bigr)
=
16a^2(1+u^2)R_{a,b}(u)>0
\end{equation*}
for every $u\in\mathbb R$. Therefore, all the hypotheses of the
preceding proposition are satisfied, and
system~\eqref{explicit-quadratic} has at most one limit cycle
surrounding the origin.

\end{enumerate}

\medskip 

We also determine its stability. Since $H(0)=3b>0$ and the leading
coefficient of $H$ is $a$, the unique real zero of $H$ has sign
opposite to that of $a$. Since $G$ is strictly increasing and
$G(0)=-a$, its unique real zero has the same sign as $a$.
Thus, the two zeros are distinct, and $A$ has exactly two simple
zeros in $(0,\pi)$. Moreover,
$A(0)=-h(0)g(0)=-a$.
Hence Theorem~\ref{quadratic-criterion} shows that the limit cycle,
if it exists, is stable for $a>0$ and unstable for $a<0$.

It remains to prove that the upper bound is sharp. Fix
$a$ with $0<|a|\leq1/3$, and let
$\Pi(s;a,b)$ denote the Poincar\'e return map of the associated
Abel equation near $s=0$. A standard recursive expansion of the
return map gives
\begin{equation*}
    \Pi(s;a,b)-s
=\pi a(b-1)s^3
+\frac{4\pi a}{3}(b-1)(2b+1)s^4+\bigl(2\pi a(3a^2-4)+O(b-1)\bigr)s^5
+O(s^6)
\end{equation*}
as $s\to0^{+}$ and $b\to1^{+}$. 
In particular, when $b=1$,
\begin{equation*}
\Pi(s;a,1)-s
=
2\pi a(3a^2-4)s^5+O(s^6).
\end{equation*}
Since $3a^2-4<0$, the origin is a weak focus of order two at
$b=1$.

To describe its bifurcation, write
\begin{equation*}
b=1+\varepsilon^2,\qquad s=\varepsilon w,
\qquad \varepsilon>0.
\end{equation*}
Dividing $\Pi$ by
$\pi a\varepsilon^5w^3$, we obtain
\begin{equation*}
\frac{
\Pi(\varepsilon w;a,1+\varepsilon^2)-\varepsilon w
}{
\pi a\varepsilon^5w^3
}
=
1-2(4-3a^2)w^2+O(\varepsilon).
\end{equation*}
At $\varepsilon=0$, the equation for a nonzero fixed point reduces to
$1-2(4-3a^2)w^2=0$,
which has the positive simple root
$w_0=\frac{1}{\sqrt{2(4-3a^2)}}$.
Since
\begin{equation*}
\left.
\frac{\partial}{\partial w}
\bigl(1-2(4-3a^2)w^2\bigr)
\right|_{w=w_0}
\neq0,
\end{equation*}
the implicit function theorem yields a positive solution
$w=w(\varepsilon)$ for all sufficiently small
$\varepsilon>0$. Hence, for every fixed
$0<|a|\leq1/3$, system~\eqref{explicit-quadratic} has a
small-amplitude limit cycle whenever $b>1$ is sufficiently close
to $1$. Its intersection with the chosen transversal satisfies
\begin{equation*}
s
=
\sqrt{\frac{b-1}{2(4-3a^2)}}+O(b-1).
\end{equation*}
By the uniqueness already proved, this limit cycle is the unique
one surrounding the origin. Therefore, the upper bound is sharp.
\end{proof}

\section*{Acknowledgements}
The first author is partially supported by the project PID2023-151974NB-I00 funded by MICIU/
AEI/10.13039/ 501100011033/FEDER, UE, and by Junta de Extremadura through project GR24042, partially funded by the European Regional Development Fund (ERDF) ``A way to make Europe''.

The second and third authors are supported by the NNSF of China (No. 12371183 and No. 124B2006).

\section*{Author Contributions}
\noindent \textbf{José Luis Bravo Trinidad:} Writing-review \& editing, Methodology.\\
\textbf{Renhao Tian:} Writing-review \& editing, Writing-original draft, Methodology.\\
\textbf{Yulin Zhao:} Writing-review \& editing, Validation, Methodology.

\section*{Availability of data and materials}
No datasets were generated or analysed during the current study.

\section*{Conflict of Interests}
The authors declare that they have no conflict of interests regarding the publication of this paper.

\end{document}